\documentclass[a4paper, 11pt]{scrartcl}
\pdfoutput=1

\title{Some Notes on Summation by Parts Time Integration Methods}
\author{Hendrik Ranocha}
\date{8th March 2019}
\titlehead{\includegraphics[width=0.5\textwidth]{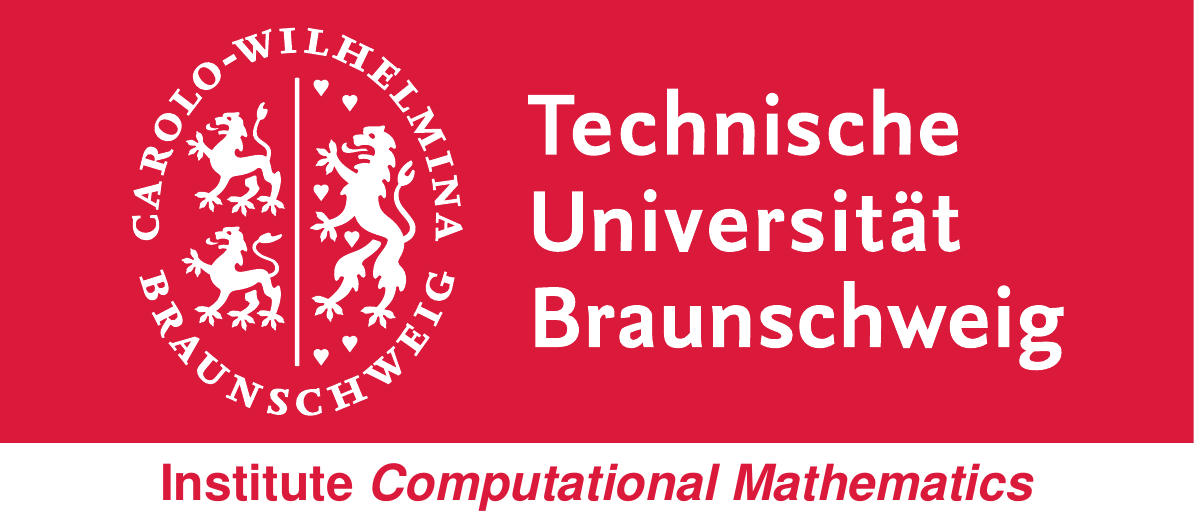}}

\usepackage[utf8]{luainputenc}
\usepackage[UKenglish]{babel}
\usepackage{csquotes}

\usepackage[a4paper, top=0.51in, bottom=0.79in, left=0.999in, right=0.99in]{geometry}

\usepackage[plainpages=false,pdfpagelabels,hidelinks,unicode]{hyperref}
\makeatletter
\hypersetup{pdfauthor={\@author}}
\hypersetup{pdftitle={\@title}}
\makeatother

\usepackage[%
  backend=biber,
  style=numeric-comp,
  giveninits=true, uniquename=init, 
  natbib=true,
  url=true,
  doi=true,
  isbn=false,
  backref=false,
  maxnames=99,
  ]{biblatex}
\usepackage{amsmath}
\allowdisplaybreaks
\usepackage{amssymb}
\usepackage{commath}
\usepackage{mathtools}
\usepackage{bbm}
\usepackage{nicefrac}

\usepackage{siunitx}

\usepackage{amsthm}
\usepackage{thmtools}
\usepackage{etoolbox}
\makeatletter
\patchcmd{\thmt@setheadstyle}
 {\bgroup\thmt@space}
 {\thmt@space}
 {}{}
\patchcmd{\thmt@setheadstyle}
 {\egroup\fi}
 {\fi}
 {}{}
\makeatother
\declaretheoremstyle[
  bodyfont=\normalfont\itshape,
  headformat=\NAME\ \NUMBER\NOTE,
]{myplain}
\declaretheoremstyle[
  headformat=\NAME\ \NUMBER\NOTE,
]{mydefinition}
\declaretheorem[style=myplain,numberwithin=section]{theorem}

\declaretheorem[style=mydefinition,numberlike=theorem]{definition}
\declaretheorem[style=mydefinition,numberlike=theorem]{remark}
\declaretheorem[style=mydefinition,numberlike=theorem]{example}

\declaretheorem[style=mydefinition,numberlike=theorem]{assumption}

\usepackage{ifluatex}
\ifluatex
  \usepackage[no-math]{fontspec}
\else
  \usepackage[T1]{fontenc}
\fi
\usepackage{newpxtext,newpxmath}

\usepackage{color}
\usepackage{graphicx}
\usepackage[small]{caption}
\usepackage{subcaption}

\usepackage{pgfplots}
\usepackage{tikzscale}
\pgfplotsset{compat=1.13}
\usepgfplotslibrary{groupplots}
\usepgfplotslibrary{external}
\begingroup\expandafter\expandafter\expandafter\endgroup
\expandafter\ifx\csname pdfsuppresswarningpagegroup\endcsname\relax
\else
\fi

\usepackage{booktabs}
\usepackage{rotating}
\usepackage{multirow}

\usepackage{multicol}
\usepackage{enumitem}

\usepackage{calc}
\usepackage{xparse}

\newcommand{\diag}{\operatorname{diag}}

\AtBeginDocument{%
  
}

\newcommand{\scp}[2]{\left\langle{#1,\, #2}\right\rangle}

\DeclarePairedDelimiterX\newset[1]\lbrace\rbrace{\setaux #1||\endsetaux}
\def\setaux#1|#2|#3\endsetaux{\if\relax\detokenize{#2}\relax #1 \else #1 \;\delimsize\vert\; #2 \fi}
\renewcommand{\set}[1]{\newset*{#1}}

\newcommand{\I}{\operatorname{I}}

\NewDocumentCommand{\ENO}{o}{%
  \IfValueTF{#1}{%
    ENO$(#1)$
  }{%
    ENO
  }%
}
\NewDocumentCommand{\mENO}{o}{%
  \IfValueTF{#1}{%
    mENO$(#1)$
  }{%
    mENO
  }%
}

\AtBeginDocument{%
  \let\rho\varrho
  \let\phi\varphi
  \let\epsilon\varepsilon
}
\newcommand{\N}{\mathbb{N}}
\newcommand{\R}{\mathbb{R}}
\renewcommand{\C}{\mathbb{C}}
\renewcommand{\Re}{\operatorname{Re}}

\newsavebox{\DelimiterBox}
\newlength{\DelimiterHeight}
\newlength{\DelimiterDepth}
\newsavebox{\ArgumentBox}
\newlength{\ArgumentHeight}
\newlength{\ArgumentDepth}
\newlength{\ResizedDelimiterHeight}
\newlength{\ResizedDelimiterDepth}

\ifluatex

\else

\fi

\begin{document}

\maketitle

\begin{abstract}
  
Some properties of numerical time integration methods using summation by parts (SBP)
operators and simultaneous approximation terms are studied. These
schemes can be interpreted as implicit Runge-Kutta methods with desirable stability
properties such as $A$-, $B$-, $L$-, and algebraic stability \cite{nordstrom2013summation,
lundquist2014sbp, boom2015high, ruggiu2018pseudo}. Here, insights into
the necessity of certain assumptions, relations to known Runge-Kutta methods,
and stability properties are provided by new proofs and counterexamples.
In particular, it is proved that a) a technical assumption is necessary since
it is not fulfilled by every SBP scheme, b) not every Runge-Kutta scheme having
the stability properties of SBP schemes is given in this way, c) the classical
collocation methods on Radau and Lobatto nodes are SBP schemes, and d) nearly
no SBP scheme is strong stability preserving.

\end{abstract}

\section{Known Results on SBP SAT Schemes}
\label{sec:known-results}

In order to solve an ordinary differential equation (ODE)
\begin{equation}
\label{eq:ODE}
  \forall t \in (0,T)\colon u'(t) = f(t, u(t)),
  \qquad
  u(0) = u_0,
\end{equation}
a grid $0 \leq \tau_1 < \dots < \tau_s \leq T$ is introduced and the numerical
solution is approximated pointwise as $u_i = u(\tau_i)$ and $f_i = f(\tau_i, u_i)$.
Summation by parts (SBP) operators can be defined as follows, cf. \cite{svard2014review,
fernandez2014review, fernandez2014generalized}.
\begin{definition}
  An SBP operator of order $p \in \N$ on $[0, T]$ consists of
  \begin{itemize}
    \item
    a discrete operator $D$ approximating the derivative $D u \approx u'$ with
    order of accuracy $p$,

    \item
    a symmetric and positive definite discrete mass/norm matrix $M$ approximating
    the $L^2$ scalar product $u^T M v \approx \int_{0}^{T} u(t) v(t) \dif t$,

    \item
    and interpolation vectors $t_L, t_R$ approximating the interpolation to the
    boundary as $t_L^T u \approx u(0)$, $t_R^T u \approx u(T)$ with order of
    accuracy at least $p$, such that
  \end{itemize}
  \begin{equation}
  \label{eq:SBP}
    M D + D^T M = t_R t_R^T - t_L t_L^T.
  \end{equation}
\end{definition}

SBP operators mimic integration by parts discretely via the summation by parts
property \eqref{eq:SBP}.
An SBP time discretisation using a simultaneous approximation term (SAT) of
\eqref{eq:ODE} with parameter $\sigma \in \R$ is \cite{nordstrom2013summation,
lundquist2014sbp, boom2015high}
\begin{equation}
  D u = f + \sigma M^{-1} t_L \bigl( u_0 - t_L^T u \bigr).
\end{equation}
Most stability results have been achieved for the choice $\sigma = 1$, i.e.
\begin{equation}
\label{eq:discretisation}
  D u = f + M^{-1} t_L \bigl( u_0 - t_L^T u \bigr).
\end{equation}
Hence, this discretisation will be considered in the following.
The numerical solution at $t=T$ is given by $t_R^T u$, where $u$ solves
\eqref{eq:discretisation}. The interval $[0,T]$ can also be partitioned into
multiple subintervals/blocks such that multiple steps of this procedure are used
sequentially.

In order to guarantee that \eqref{eq:discretisation} can be solved for a dissipative
linear scalar problem, the following assumption is introduced \cite{nordstrom2013summation}.
\begin{assumption}
\label{assumption}
  For $\sigma > \frac{1}{2}$, all eigenvalues of $D + \sigma M^{-1} t_L t_L^T$
  have strictly positive real part.
\end{assumption}
The following characterisation of \eqref{eq:discretisation} as Runge-Kutta method
has been developed in \cite{boom2015high}.
\begin{theorem}
\label{thm:SBP-SAT-as-RK}
  If assumption~\ref{assumption} is satisfied, \eqref{eq:discretisation} is equivalent
  to an implicit Runge-Kutta method with the following Butcher coefficients,
  where $1$ denotes also the vector $(1, \dots, 1)^T \in \R^s$.
  \begin{equation}
  \label{eq:SBP-SAT-as-RK}
    A = \frac{1}{T} (D + M^{-1} t_L t_L^T)^{-1} = \frac{1}{T} (M D + t_L t_L^T)^{-1} M,
    \quad
    b = \frac{1}{T} M 1,
    \quad
    c = \frac{1}{T} (\tau_1, \dots, \tau_s)^T.
  \end{equation}
\end{theorem}
The factor $\frac{1}{T}$ is needed since Runge-Kutta coefficients are normalised
to the interval $[0,1]$.

In order to make this article sufficiently self-contained, some classical stability
properties of Runge-Kutta methods will be recalled briefly, cf.
\cite[sections IV.3 and IV.12]{hairer2010solving}.
The absolute value of solutions of the scalar linear ODE $u'(t) = \lambda u(t)$,
$u(0) = u_0 \in \C$, $\lambda \in \C$, cannot increase if $\Re \lambda \leq 0$.
The numerical solution after one time step of a Runge-Kutta method with Butcher
coefficients $A, b, c$ is $u_+ = R(\lambda \, \Delta t) u_0$, where
\begin{equation}
  R(z)
  =
  1 + z b^T (\I - z A)^{-1} 1
  =
  \frac{\det(\I - z A + z 1 b^T)}{\det(\I - z A)}
\end{equation}
is the \emph{stability function} of the Runge-Kutta method. The stability property
is mimicked discretely as $\abs{u_+} \leq \abs{u_0}$ if $\abs{R(\lambda \, \Delta t)}
\leq 1$.
\begin{definition}
  A Runge-Kutta method with stability function $R(z)$ is $A$-stable, if
  $\abs{R(z)} \leq 1$ for all $z \in \C$ with $\Re(z) \leq 0$.
  The method is $L$-stable, if it is $A$-stable and $\lim_{z \to \infty} R(z) = 0$.
\end{definition}
Hence, $A$-stable methods are stable for every time step $\Delta t > 0$ and
$L$-stable methods damp out stiff components corresponding to $\lambda = -x$
with large $x \in \R$ sufficiently fast.

Another classical stability property is connected with possibly nonlinear problems
\eqref{eq:ODE} in Hilbert spaces satisfying a \emph{one-sided Lipschitz condition}
\begin{equation}
\label{eq:one-sided-Lipschitz}
  \forall t,u,v\colon \quad
  \scp{f(t,u) - f(t,v)}{u - v} \leq \nu \norm{u - v}^2,
\end{equation}
where $\nu \in \R$ is the \emph{one-sided Lipschitz constant} of $f$. This condition
gives some bounds on the growth rate of the difference between two solutions. In
particular, the distance between two solutions cannot increase if $\nu \leq 0$.
\begin{definition}
  A Runge-Kutta method is $B$-stable, if the contractivity condition
  \eqref{eq:one-sided-Lipschitz} with $\nu \leq 0$ implies
  $\norm{u_+ - v_+} \leq \norm{u_0 - v_0}$ for all $\Delta t > 0$.
\end{definition}

The following stability properties have been obtained in \cite{lundquist2014sbp,
boom2015high}.
\begin{theorem}
\label{thm:SBP-SAT-stability}
  Suppose that assumption~\ref{assumption} holds.
  Then, the SBP SAT scheme \eqref{eq:discretisation} is $A$- and $L$-stable.
  If the mass matrix $M$ is diagonal, the scheme is also $B$-stable.
\end{theorem}

\section{Assumptions and Algebraic Stability}
\label{sec:assumptions_and_stability}

In this section, the new results of this short note concerning the necessity
of assumption~\ref{assumption} and the necessity of an SBP SAT form for stability
properties guaranteed by Theorem~\ref{thm:SBP-SAT-stability} are presented.

\subsection{Assumption on Eigenvalues of \texorpdfstring{$D + \sigma M^{-1} t_L t_L^T$}{D + σ M⁻¹ tₗ tₗᵀ}}

Assumption~\ref{assumption} has been proved for classical second order SBP operators
in \cite{nordstrom2013summation} and for SBP operators on Gauss, Radau, and Lobatto
quadrature nodes in \cite{ruggiu2018pseudo}. It has been examined numerically
for other classical finite difference SBP operators in \cite{nordstrom2013summation}.
Since assumption~\ref{assumption} holds for all known SBP SAT schemes investigated
in \cite{nordstrom2013summation, lundquist2014sbp, boom2015high, ruggiu2018pseudo},
it is interesting to know whether it follows from properties of SBP operators.
\begin{theorem}
\label{thm:assumption}
  There are SBP operators that do not satisfy assumption~\ref{assumption}.
\end{theorem}
\begin{proof}
  Consider the operators
  \begin{equation}
    D =
    \begin{pmatrix}
      -2 & 1 & 1 & 0 \\
      -1 & 0 & 0 & 1 \\
      -1 & 0 & 0 & 1 \\
      0 & -1 & -1 & 2 \\
    \end{pmatrix},
    \quad
    M = \frac{1}{4}
    \begin{pmatrix}
      1 & 0 & 0 & 0 \\
      0 & 1 & 0 & 0 \\
      0 & 0 & 1 & 0 \\
      0 & 0 & 0 & 1 \\
    \end{pmatrix},
    \quad
    t_L =
    \begin{pmatrix}
      1 \\ 0 \\ 0 \\ 0
    \end{pmatrix},
    \quad
    t_R =
    \begin{pmatrix}
      0 \\ 0 \\ 0 \\ 1
    \end{pmatrix},
  \end{equation}
  on the uniform grid with four nodes $0, \frac{1}{3}, \frac{2}{3}, 1$ in $[0,1]$.
  The SBP property \eqref{eq:SBP} is satisfied, $t_L$ and $t_R$ are exact, and
  $D$ is a first order accurate SBP derivative operator.
  However, $(D + \sigma M^{-1} t_L t_L^T) u = 0$ for $u = (0, -1, 1, 0)^T$.
  Thus, zero is an eigenvalue of $D + \sigma M^{-1} t_L t_L^T$ for all
  $\sigma \in \R$.
\end{proof}

\subsection{Algebraic Stability}

Many stability properties such as $A$- and $B$-stability are satisfied if the
following algebraic criterion is fulfilled by the coefficients of a Runge-Kutta
method \cite[Theorem~12.4]{hairer2010solving}.
\begin{definition}
  A Runge-Kutta method with Butcher coefficients $A, b, c$ is algebraically stable,
  if $\forall i\colon b_i \geq 0$ and the matrix $\diag(b) A + A^T \diag(b) - b b^T$
  is positive semidefinite.
\end{definition}
It has been noted in \cite{boom2015high} that an SBP SAT scheme \eqref{eq:discretisation}
with diagonal $M$ is algebraically stable, since the nodes $\tau_i$ are pairwise
distinct, i.e. the corresponding Runge-Kutta method is nonconfluent. In that case,
$B$- and algebraic stability are equivalent \cite[Corollary~12.14]{hairer2010solving}.
This can also be proved directly, cf. \cite[Theorem~5.8]{boom2015high}.

It is interesting to know whether all Runge-Kutta methods with stability properties
guaranteed by Theorem~\ref{thm:SBP-SAT-stability} can be constructed as SBP SAT
schemes.
Since those schemes are $L$-stable, the classical Gauss collocation schemes
(which are not $L$-stable) cannot be constructed in this way, cf. \cite{boom2015high}.
However, there is
\begin{theorem}
\label{thm:alg-L-stable-is-SBP}
  Consider a Runge-Kutta method and the statements
  \begin{enumerate}[label=\roman*)]
    \item \label{itm:RKM-properties}
    The Runge-Kutta method is $A$-, $L$-, $B$-, and algebraically stable with
    pairwise distinct nodes $c_i \in [0,1]$, only positive quadrature weights
    $b_i$, and invertible matrix $A$.

    \item \label{itm:SBP-SAT}
    The Runge-Kutta method is given via Theorem~\ref{thm:SBP-SAT-as-RK} by
    SBP SAT schemes \eqref{eq:discretisation} with at least first order accurate
    operators satisfying assumption~\ref{assumption}.
  \end{enumerate}
  Theorem~\ref{thm:SBP-SAT-stability} and the preceeding discussion show that
  \ref{itm:SBP-SAT} and ``$M$ is diagonal'' imply \ref{itm:RKM-properties}.
  However, \ref{itm:RKM-properties} does not imply \ref{itm:SBP-SAT}.
\end{theorem}
\begin{proof}
  The following example has been constructed using the $W$-transformation
  \cite[Sections~IV.5, IV.13, and~IV.14]{hairer2010solving}.
  Consider the Runge-Kutta method with coefficients
  \begin{equation}
    A = \frac{1}{48}
    \begin{pmatrix}
      27 & -33-6 \sqrt{6} & -3 & 9+6 \sqrt{6} \\
      -7+2 \sqrt{6} & 33 & -9-2 \sqrt{6} & -1 \\
      7 & 3+2 \sqrt{6} & 33 & -11-2 \sqrt{6} \\
      21-6 \sqrt{6} & 21 & -21+6 \sqrt{6} & 27 \\
    \end{pmatrix},
    \quad
    b = \frac{1}{8}
    \begin{pmatrix}
      1 \\ 3 \\ 3 \\ 1
    \end{pmatrix},
    \quad
    c = \frac{1}{3}
    \begin{pmatrix}
      0 \\ 1 \\ 2 \\ 3
    \end{pmatrix}.
  \end{equation}
  Then, the algebraic stability matrix $\diag(b) A + A^T \diag(b) - b b^T$
  has the eigenvalues $\frac{5}{8}$, $\frac{3}{8}$, and zero (twofold). Hence, the
  Runge-Kutta method is algebraically stable (because $b_i > 0$ is satisfied additionally)
  and therefore also $A$- and $B$-stable. Its stability functions
  \begin{equation}
    R(z)
    =
    \frac{\det(\I - z A + z 1 b^T)}{\det(\I - z A)}
    =
    \frac{12 - 18 z + 3 z^2 + 3 z^3}{12 - 30 z + 27 z^2 - 11 z^3 + 2 z^4},
  \end{equation}
  fulfils $\lim_{z\to\infty} R(z) = 0$. Thus, the scheme is also $L$-stable.

  It suffices to consider $T=1$. If the scheme is given by an SBP SAT method
  \eqref{eq:discretisation} via Theorem~\ref{thm:SBP-SAT-as-RK},
  $b = M 1$ and $A = (D + M^{-1} t_L t_L^T)^{-1}$.
  The SBP property \eqref{eq:SBP} yields
  $A^{-1} = - M^{-1} D^T M + M^{-1} t_R t_R^T$.
  Because of consistency, $D 1 = 0$ and $t_R^T 1 = 1$. Hence, $1^T M A^{-1} = t_R^T$.
  Inserting $M 1 = b$ results in
  \begin{equation}
  \label{eq:proof-tR}
    t_R
    =
    A^{-T} b
    =
    \frac{1}{16} \left(
      - 4 + \sqrt{6}, -\sqrt{6}, 12 - \sqrt{6}, 8 + \sqrt{6}
    \right)^T.
  \end{equation}
  Similarly, consistency of $D$ and $t_L$ implies
  \begin{equation}
  \label{eq:proof-tL}
    A^{-1} 1 = (D + M^{-1} t_L t_L^T) 1 = M^{-1} t_L
    \iff
    t_L = M A^{-1} 1.
  \end{equation}
  $t_R$ defined by \eqref{eq:proof-tR} is first order
  accurate, i.e. $t_R^T 1 = 1$ and $t_R^T c = 1$.
  The same accuracy of $t_L$ requires
  \begin{equation}
  \label{eq:proof-order}
    t_L^T 1 = 1,
    \quad
    t_L^T c = 0.
  \end{equation}
  Because of \eqref{eq:proof-tL}, $D$ can be written as
  \begin{equation}
  \label{eq:proof-D}
    D = A^{-1} - M^{-1} t_L t_L^T = A^{-1} - A^{-1} 1 t_L^T.
  \end{equation}
  Since $M \in \mathbb{R}^{4 \times 4}$ should be symmetric, it is determined
  by ten real parameters, e.g. $M_{11}$, $M_{12}$, $M_{13}$, $M_{14}$, $M_{22}$,
  $M_{23}$, $M_{24}$, $M_{33}$, $M_{34}$, $M_{44}$.
  $t_R$ is given explicitly by \eqref{eq:proof-tR}, $t_L$ depends linearly on
  $M$ via \eqref{eq:proof-tL}, and $D$ is given via an affine-linear function of
  $M$ in \eqref{eq:proof-D}.

  The accuracy conditions \eqref{eq:proof-order} are linear in $t_L$ and hence
  linear in $M$. They can be used to eliminate two parameters, e.g. $M_{11}$
  and $M_{12}$. Then, the SBP property \eqref{eq:SBP} is a system of 16 equations
  that are quadratic in the parameters $M_{ij}$.
  This system can be solved uniquely, which has been verified using the function
  \texttt{Reduce} of Mathematica~\cite{mathematica10}.
  For this unique solution, one eigenvalue of $M$ is zero. Thus, $M$ is not
  positive definite, in contradiction to the assumptions.
\end{proof}

\section{Classical Collocation Methods}
\label{sec:classical-collocation-methods}

In \cite{boom2015high}, it has been shown that the SBP SAT scheme with Lobatto
quadrature on four nodes corresponds to the classical Lobatto~IIIC method with
$s=4$. It has been mentioned that this is similar for the Radau~IA and
Radau~IIA schemes. However, to the authors knowledge, no general proof of this
result has been given up to now. To prove it, the classical conditions
\begin{align}
\label{eq:C}
  C(\eta) &\colon
  & \sum_{j=1}^s a_{i,j} c_j^{q-1} &= \frac{1}{q} c_i^q,
  && i \in \set{1, \dots, s},\; q \in \set{1, \dots, \eta},
  \\
\label{eq:D}
  D(\zeta) &\colon
  & \sum_{i=1}^s b_i c_i^{q-1} a_{i,j} &= \frac{1}{q} b_j (1 - c_j^q),
  && j \in \set{1, \dots, s},\; q \in \set{1, \dots, \zeta},
\end{align}
will be used.
\begin{theorem}
\label{thm:classical-collocation-methods}
  The SBP SAT scheme \eqref{eq:discretisation} using left Radau, right Radau, or
  Lobatto quadrature correspond to the classical Radau~IA, Radau~IIA, or Lobatto~IIIC
  Runge-Kutta methods for all orders of accuracy.
\end{theorem}
\begin{proof}
  It suffices to consider the case $T = 1$, i.e. the time interval $[0,1]$.

  The weights and nodes of the left Radau quadrature (left endpoint $0$ included)
  are the weights $b_i$ and nodes $c_i$ of the Radau~IA method. The matrix $A$
  of the Radau~IA method is determined uniquely by the condition $D(s)$,
  i.e. $D(\zeta)$ with $\zeta = s$ in \eqref{eq:D}
  \cite[section~IV.5]{hairer2010solving}. Hence, it suffices to prove that the
  SBP SAT method satisfies $D(s)$, which can be written using $M = \diag(b)$ as
  \begin{equation}
    A^T M c^{q-1} = \frac{1}{q} M (1 - c^q)
    \iff
    q M c^{q-1} = A^{-T} M (1 - c^q),
  \end{equation}
  where the exponentiation $c^q$ is performed pointwise. Inserting $A$ from
  \eqref{eq:SBP-SAT-as-RK} yields
  \begin{equation}
    q M c^{q-1} = (D^T M + t_L t_L^T) (1 - c^q).
  \end{equation}
  This is equivalent to
  \begin{equation}
    \forall v\colon \quad
    q v^T M c^{q-1} = v^T (D^T M + t_L t_L^T) (1 - c^q),
  \end{equation}
  where $v$ is any polynomial of degree $\leq s-1$, evaluated at the nodes $c_i$.
  Since the left endpoint~$0$ is included,
  \begin{equation}
    v^T t_L t_L^T (1 - c^q) = v(0) \, (1 - 0^q) = v(0).
  \end{equation}
  The Radau quadrature is exact for polynomials of degree $\leq 2s-2$. Hence,
  for every $q \in \set{1, \dots, s}$,
  \begin{equation}
    q v^T M c^{q-1}
    =
    q \int_0^1 v(t) t^{q-1} \dif t
  \end{equation}
  and (using integration by parts)
  \begin{equation}
    v^T D^T M (1 - c^q)
    =
    \int_0^1 v'(t) (1 - t^q) \dif t
    =
    - v(0) + q \int_0^1 v(t) t^{q-1} \dif t,
  \end{equation}
  proving $D(s)$.

  The weights and nodes of the right Radau quadrature (right endpoint $1$ included)
  are the weights $b_i$ and nodes $c_i$ of the Radau~IIA method. The matrix $A$
  of the Radau~IIA method is determined uniquely by the condition $C(s)$,
  i.e. $C(\eta)$ with $\eta = s$ in \eqref{eq:C}
  \cite[section~IV.5]{hairer2010solving}. Hence, it suffices to prove that the
  SBP SAT method satisfies $C(s)$, which can be written using $M = \diag(b)$ as
  \begin{equation}
    A c^{q-1} = \frac{1}{q} c^q
    \iff
    q M c^{q-1} = M A^{-1} c^q,
  \end{equation}
  where the exponentiation $c^q$ is again performed pointwise. Inserting $A$ from
  \eqref{eq:SBP-SAT-as-RK}, this is equivalent to
  \begin{equation}
    \forall v\colon \quad
    q v^T M c^{q-1} = v^T (M D + t_L t_L^T) c^q,
  \end{equation}
  where $v$ is any polynomial of degree $\leq s-1$, evaluated at the nodes $c_i$.
  Using the SBP property \eqref{eq:SBP}, this can be rewritten as
  \begin{equation}
    \forall v\colon \quad
    q v^T M c^{q-1} = v^T (-D^T M + t_R t_R^T) c^q.
  \end{equation}
  Since the right endpoint~$1$ is included,
  \begin{equation}
    v^T t_R t_R^T c^q = v(1) \, 1^q = v(1).
  \end{equation}
  Using the exactness of the Radau quadrature for polynomials of degree $\leq 2s-2$,
  for every $q \in \set{1, \dots, s}$,
  \begin{equation}
    q v^T M c^{q-1}
    =
    q \int_0^1 v(t) t^{q-1} \dif t
  \end{equation}
  and (using integration by parts)
  \begin{equation}
    - v^T D^T M c^q
    =
    - \int_0^1 v'(t) t^q \dif t
    =
    - v(1) + q \int_0^1 v(t) t^{q-1} \dif t,
  \end{equation}
  proving $C(s)$.

  Finally, the weights and nodes of the Lobatto quadrature (left and right endpoints
  $0,1$ included) are the weights $b_i$ and nodes $c_i$ of the Lobatto~IIIC method.
  The matrix $A$ of the Lobatto~IIIC method is determined uniquely by the condition
  $C(s-1)$ and $a_{i,1} = b_1, i \in \set{1, \dots, s}$ \cite[section~IV.5]{hairer2010solving}.
  Since the order of accuracy of the SBP operator is $s-1$, $C(s-1)$ is satisfied
  \cite[Lemma~5.3]{boom2015high}. This can also be proved using similar manipulations
  as above. Hence, it remains to show $a_{i,1} = b_1, i \in \set{1, \dots, s}$.
  Since $D$ is exact for constants, $t_L = (1, 0, \dots, 0)^T$, and
  $M = \diag(b_1, \dots, b_s)$,
  \begin{equation}
    (D + M^{-1} t_L t_L^T) 1 = 0 + M^{-1} t_L = b_1^{-1} t_L.
  \end{equation}
  Therefore,
  $ 
    (a_{i,1})_{i=1}^s = A t_L = (D + M^{-1} t_L t_L^T)^{-1} t_L = b_1 1,
  $ 
  proving $a_{i,1} = b_1, i \in \set{1, \dots, s}$.
\end{proof}

\section{Strong Stability Preservation}
\label{sec:SSP}

Another desirable stability property of time integration methods is that they
are strong stability preserving (SSP), i.e. that they preserve convex stability
properties of the explicit Euler method \cite{gottlieb2011strong}.
\begin{definition}
  A numerical time integration method is called strongly stable for a given
  convex functional $\eta$ if $\eta(u_+) \leq \eta(u_0)$, possibly using some
  time step restriction of the form $0 < \Delta t \leq \Delta t_\mathrm{max}$.

  A numerical time integration method is called strong stability preserving with
  SSP coefficient $c > 0$, if $\eta(u_+) \leq \eta(u_0)$ for all time steps
  $0 < \Delta t \leq c \, \Delta t_E$ whenever the explicit Euler method is strongly
  stable for the convex functional $\eta$ and time steps $0 < \Delta t \leq \Delta t_E$.
\end{definition}
Typical convex functionals $\eta$ considered for SSP methods are the norm in a
Hilbert space for dissipative operators or the total variation seminorm for
semidiscretisations of scalar conservation laws.
\begin{theorem}
\label{thm:SSP}
  No SBP SAT scheme \eqref{eq:discretisation} whose SBP operator has a diagonal
  norm matrix, satisfies assumption~\ref{assumption}, and
  \begin{enumerate}[label=\alph*)]
    \item is either at least second order accurate
    \item or is at least first order accurate and contains at least one of the
          end points $0, 1$ in the nodes $c_i$
  \end{enumerate}
  can be strong stability preserving.
\end{theorem}
\begin{proof}
  An SSP scheme must satisfy $\forall i,j \in \set{1, \dots, s}\colon a_{i,j} \geq 0$
  \cite[Observation~5.2]{gottlieb2011strong}.

  If the SBP operator is at least second order accurate, the corresponding Runge-Kutta
  method satisfies $C(2)$ \cite[Lemma~5.3]{boom2015high}, i.e.
  $\sum_{j=1}^s a_{i,j} = c_i$ and
  $\sum_{j=1}^s a_{i,j} c_j = \frac{1}{2} c_i^2$
  for $i \in \set{1, \dots, s}$.
  Subtracting the second equation from the first one multiplied by $c_i$ yields
  \begin{equation}
    \sum_{j=1}^s a_{i,j} (c_i - c_j) = \frac{1}{2} c_i^2,
    \quad
    i \in \set{1, \dots, s}.
  \end{equation}
  If $a_{i,j}$ were non-negative, the left hand side would be non-positive
  for $i=1$ (since $c_j \geq c_1$) and thus zero. Hence, the first row of $A$
  would be zero, which is impossible, because $A$ is invertible.

  If the SBP operator is at least first order accurate, the corresponding Runge-Kutta
  method satisfies $C(1)$ and $D(1)$ \cite[Lemma~5.3 and Lemma~5.4]{boom2015high}, i.e.
  \begin{equation}
    \sum_{j=1}^s a_{i,j} = c_i,
    \; i \in \set{1, \dots, s},
    \qquad\qquad
    \sum_{i=1}^s b_i a_{i,j} = b_j (1 - c_j),
    \;
    j \in \set{1, \dots, s}.
  \end{equation}
  If the left endpoint $0 = c_1$ is contained in the nodes, non-negativity of
  all $a_{i,j}$ and $C(1)$ imply $\forall j \in \set{1,\dots,s}\colon a_{1,j} = 0$.
  Similarly, if the right endpoint $c_s = 1$ is contained in the nodes, non-negativity of
  all $a_{i,j}$ and $D(1)$ imply $\forall i \in \set{1,\dots,s}\colon a_{i,s} = 0$.
  But $A$ cannot have a zero row or column because it is invertible.
\end{proof}
\begin{remark}
\label{rem:SSP}
  Classical finite difference SBP operators and those based on Radau or Lobatto
  quadrature include at least one endpoint and can thus not result in SSP schemes.
  The SBP SAT scheme \eqref{eq:discretisation} on two Gauss nodes does not contain
  an endpoint and has a first order accurate derivative operator. Nevertheless,
  the scheme is not SSP, since the corresponding matrix
  $A$ has a negative entry.
\end{remark}

\begin{example}
\label{ex:SSP}
  There is a first order accurate SBP operator with diagonal norm matrix not
  including any boundary node such that the resulting Runge-Kutta method given
  by Theorem~\ref{thm:SBP-SAT-as-RK} is SSP. Indeed, choose $T=1$ and
  \begin{equation}
  \begin{gathered}
    D = \frac{1}{128}
    \begin{pmatrix}
      -2079 & 3646 & -1567 \\
      271 & -1054 & 783 \\
      -479 & 446 & 33 \\
    \end{pmatrix},
    \quad
    M = \frac{1}{4}
    \begin{pmatrix}
      1 & 0 & 0 \\
      0 & 2 & 0 \\
      0 & 0 & 1 \\
    \end{pmatrix},
    \quad
    t_L =
    \begin{pmatrix}
      3 \\ -3 \\ 1
    \end{pmatrix},
    \quad
    t_R = \frac{1}{16}
    \begin{pmatrix}
      -15 \\ 14 \\ 17
    \end{pmatrix},
    \\
    c = \frac{1}{4}
    \begin{pmatrix}
      1 \\ 2 \\ 3
    \end{pmatrix},
    \quad
    A = (D + M^{-1} t_L t_L^T)^{-1} = \frac{1}{\num{20000}}
    \begin{pmatrix}
      2725 & 2180 & 95 \\
      4390 & 5512 & 98 \\
      3495 & 6796 & 4709 \\
    \end{pmatrix}.
  \end{gathered}
  \end{equation}
  The operators $D, t_L, t_R$ are exact for polynomials of degree one,
  assumption~\ref{assumption} has been verified numerically for $\sigma \in (1/2,2)$,
  $A$ and $b$ have only non-negative entries, and the scheme is strong stability
  preserving with SSP coefficient $\approx 1.35$, computed using NodePy \cite{nodepy}.
\end{example}

\section*{Acknowledgements}

This work was supported by the German Research Foundation (DFG, Deutsche
Forschungsgemeinschaft) under Grant SO~363/14-1.
We acknowledge support by the German Research Foundation and the Open Access
Publication Funds of the Technische Universität Braunschweig.
The author would like to thank the anonymous reviewers for their helpful comments
and valuable suggestions to improve this article.

\appendix

\printbibliography

\end{document}